\documentclass[reqno]{amsart}

\usepackage{tikz}
\usepackage{hyperref}

\newcommand*{\mailto}[1]{\href{mailto:#1}{\nolinkurl{#1}}}
\newcommand{\arxiv}[1]{\href{http://arxiv.org/abs/#1}{arXiv: #1}}
\newcommand{\doi}[1]{\href{http://dx.doi.org/#1}{doi: #1}}

\newtheorem{theorem}{Theorem}[section]
\newtheorem{definition}[theorem]{Definition}
\newtheorem{lemma}[theorem]{Lemma}

\newtheorem{proposition}[theorem]{Proposition}
\newtheorem{corollary}[theorem]{Corollary}

\newtheorem{remark}[theorem]{Remark}

\newcommand{\R}{{\mathbb R}}

\newcommand{\Z}{{\mathbb Z}}
\newcommand{\C}{{\mathbb C}}

\newcommand{\E}{\mathrm{e}}
\newcommand{\I}{\mathrm{i}}

\newcommand{\tr}{\mathrm{tr}}

\newcommand{\loc}{\mathrm{loc}}

\newcommand{\cc}{\mathrm{c}}

\newcommand{\be}{\begin{equation}}
\newcommand{\ee}{\end{equation}}
\newcommand{\nrc}{\gamma}

\newcommand{\OO}{\mathcal{O}}
\newcommand{\oo}{o}

\newcommand{\ledot}{\,\cdot\,}
\newcommand{\redot}{\cdot\,}

\newcommand{\Iso}[1]{\mathrm{Iso}(#1)}

\newcommand{\dip}{\upsilon}
\newcommand{\D}{\mathcal{D}}
\newcommand{\Pe}{\mathcal{P}}
\newcommand{\per}{\ell}
\newcommand{\inds}{\mathcal{I}}
\newcommand{\To}{\mathcal{T}}

\renewcommand{\labelenumi}{(\roman{enumi})}
\renewcommand{\theenumi}{\roman{enumi}}
\numberwithin{equation}{section}

\begin{document}

\title[Periodic inverse spectral problem]{The inverse spectral problem for periodic conservative multi-peakon solutions of the Camassa--Holm equation}

\author[J.\ Eckhardt]{Jonathan Eckhardt}
\address{Faculty of Mathematics\\ University of Vienna\\ Oskar-Morgenstern-Platz 1\\ 1090 Wien\\ Austria}
\email{\mailto{jonathan.eckhardt@univie.ac.at}}
\urladdr{\url{http://homepage.univie.ac.at/jonathan.eckhardt/}}

\author[A.\ Kostenko]{Aleksey Kostenko}
\address{Faculty of Mathematics and Physics\\ University of Ljubljana\\ Jadranska 19\\ 1000 Ljubljana\\ Slovenia\\ and Faculty of Mathematics\\ University of Vienna\\ Oskar-Morgenstern-Platz 1\\ 1090 Wien\\ Austria}
\email{\mailto{Aleksey.Kostenko@fmf.uni-lj.si};\ \mailto{Oleksiy.Kostenko@univie.ac.at}}
\urladdr{\url{http://www.mat.univie.ac.at/~kostenko/}}

\thanks{{\it Research supported by the Austrian Science Fund (FWF) under Grants No.~P29299 (J.E.) and P28807 (A.K.)}}


\begin{abstract}
  We solve the inverse spectral problem associated with periodic conservative multi-peakon solutions of the Camassa--Holm equation.
  The corresponding isospectral sets can be identified with finite dimensional tori.  
\end{abstract}

\maketitle

\section{Introduction}
 
  In the course of the last few decades, the Camassa--Holm equation
   \begin{equation}\label{eqnCH}
   u_{t} -u_{xxt}  = 2u_x u_{xx} - 3uu_x + u u_{xxx} 
  \end{equation}
  has become one of the most intensively studied partial differential equations. 
  Due to the vast amount of literature devoted to various aspects of this equation, we only refer to a brief selection of articles  \cite{besasz00, bkst, brco07, co97, coes98, como00, cost00, ektIMN, GH08, hora07, le05b, mc04, mc03, xizh00} containing further information. 
     The relevance of the Camassa--Holm equation stems from the fact that it arises as a model for unidirectional wave propagation on shallow water \cite{caho93, cola09, jo02}. 
  Probably the most distinctive feature of this equation however  is that it allows for smooth solutions to blow up in finite time in a way that resembles wave-breaking. 
 This behavior has been described in detail and is known to happen only when the quantity $\omega=u-u_{xx}$ is indefinite; see  \cite{coes98, mc04, mc03}. 
 As opposed to the rather tame definite case (which shows similarities with the Korteweg--de Vries equation \cite{besasz98, le04, mc03b}), sign-changes of $\omega$ inflict serious difficulties (as visible from the expositions in \cite{comc99, ka06, le05} for example). 

On the other side, the Camassa--Holm equation is known to be formally completely integrable in the sense that there is an associated isospectral problem  
\begin{align}\label{eqnISPcla}
 -f'' + \frac{1}{4} f = z\, \omega f, 
\end{align}
where $z$ is a complex spectral parameter.  
As solving corresponding inverse spectral problems is essentially equivalent to solving initial value problems for the Camassa--Holm equation, it is not surprising that the encountered complications due to wave-breaking for indefinite $\omega$ reoccur within this context as well. 
 In fact, results on inverse spectral theory in this case remained rather scarce \cite{besasz00, be04, bebrwe08, bebrwe12, bebrwe15, LeftDefiniteSL, IsospecCH} for some time. 
 Only recently, we proposed a way to overcome these problems by means of generalizing the isospectral problem  \cite{ConservCH, ConservMP, IndefiniteString}, which was motivated by work on the indefinite moment problem of M.\ G.\ Krein and H.\ Langer \cite{krla79}.

  In this article, we are interested in the isospectral problem for the Camassa--Holm equation in the periodic situation. 
  More precisely, we will solve the inverse spectral problem corresponding to {\em periodic conservative multi-peakon solutions} of the Camassa--Holm equation, which can be regarded as one kind of periodic finite gap solution.
  These are solutions of the particular form 
  \begin{align}
    u(x,t) = \sum_{k\in\Z} \sum_{n=1}^N p_n(t) \E^{-|x-q_n(t) - k\per|}
  \end{align}
  so that the corresponding quantity $\omega$ is a Borel measure with discrete support for every fixed time. 
  Of course, due to their lack of regularity, they have to be interpreted as suitable weak solutions. 
  The non-periodic counterpart to these solutions (that is, with vanishing spatial asymptotics) is rather well studied \cite{besasz00, besasz01, coes98, ConservMP, hora07b, wa06} and allows explicit formulas for the appearing coefficients.
    It is remarkable that most of the essential properties of the Camassa--Holm equation are already observable for this class of solutions. 
    
  Previous literature on periodic inverse spectral problems for~\eqref{eqnISPcla} is very scarce and always restricted to the case when $\omega$ is a strictly positive continuous function; see \cite{baklko03, co98, ko04}.
  Somewhat related, smooth finite gap solutions of the Camassa--Holm equation have been studied in \cite{comc99, CH2Real, GH08}.
  In the context of periodic multi-peakon solutions, it has been shown that the coefficient $\omega$ can be recovered in terms of the spectral data, which led to a representation of these solutions in terms of  theta functions  \cite{alcafehoma01, alfe01, besasz02, besasz05}. 
  We will complement these results here by actually solving the corresponding periodic inverse spectral problem via a suitable generalization of the differential equation~\eqref{eqnISPcla}, which directly corresponds to the notion of global conservative solutions of the Camassa--Holm equation in \cite{brco07, hora07}.
  
 A natural phase space for periodic conservative solutions of the Camassa--Holm equation (respectively its two-component generalization) is given by the following definition (compare \cite{grhora13, hora08b}). 
 Here and henceforth in this article, we will always suppose that $\per>0$ is an arbitrary but fixed period length.

 \begin{definition}
The set $\D$ consists of all pairs $(u,\mu)$ such that $u$ is an $\per$-periodic, real-valued function in $H^1_\loc(\R)$ and $\mu$ is an $\per$-periodic, non-negative Borel measure on $\R$ with
\begin{align}\label{eqnmuac}
 \mu(B) \geq \int_B u(x)^2 + u'(x)^2\, dx 
\end{align}
for every Borel set $B\subseteq\R$.
\end{definition}

Given a pair $(u,\mu)$ in $\D$, we define an associated distribution $\omega$ in $H^{-1}_\loc(\R)$ via 
\begin{align}\label{eqnDefomega}
 \omega(h) = \int_\R u(x)h(x)dx + \int_\R u'(x)h'(x)dx, \quad h\in H^1_\cc(\R),
\end{align}
so that $\omega = u - u''$ in a distributional sense, and a non-negative Borel measure $\dip$ on $\R$ in such a way that   
\begin{align}\label{eqnDefdip}
  \mu(B) = \dip(B) + \int_B u(x)^2 + u'(x)^2\, dx  
\end{align}
for every Borel set $B\subseteq\R$.
 Here, we denote with $H^1_\cc(\R)$ the subspace of functions in $H^1_\loc(\R)$ with compact support. 
 Let us point out that it is always possible to uniquely recover the pair $(u,\mu)$ from the distribution $\omega$ and the Borel measure $\dip$.
 With these definitions, we now consider the differential equation 
  \begin{align}\label{eqnSP}
  -f'' + \frac{1}{4} f = z\, \omega f + z^2 \dip f, 
 \end{align}
 where $z$ is a complex spectral parameter. 
 As the coefficients may be genuine distributions, this differential equation has to be interpreted in a suitable weak sense; see \cite{ConservCH, IndefiniteString, CHPencil, gewe14, ss03}.
 The corresponding spectral problem will serve as an isospectral problem for global conservative solutions of the Camassa--Holm equation as well as the two-component Camassa--Holm system; compare \cite{ConservCH, LagrangeCH, ConservMP}.
 Since we are interested in the multi-peakon case, we define the corresponding phase space next.

\begin{definition}
 The set $\Pe$ consists of all pairs $(u,\mu)$ in $\D$ such that the topological supports of the distribution $\omega$ and the Borel measure $\dip$ are discrete sets. 
\end{definition}

 We notice that for pairs $(u,\mu)$ in $\Pe$, the distribution $\omega$ can always be represented by an $\per$-periodic, real-valued Borel measure on $\R$, which we will denote with $\omega$ as well for notational simplicity. 

 Our paper will start with the discussion of the periodic spectral problem when the pair $(u,\mu)$ belongs to $\Pe$, including the introduction of necessary spectral quantities like the Floquet discriminant.
 We then continue with the Dirichlet spectral problem and, in particular, establish some important properties of an associated Weyl--Titchmarsh function.
 In addition, we derive trace formulas for all these spectra, which play a crucial role when solving initial value problems for the Camassa--Holm equation. 
 Subsequently, we then proceed to solve the inverse problem for the Dirichlet spectrum before solving the periodic inverse spectral problem.
 It turns out that isospectral sets can be identified with finite dimensional tori (a fact which is certainly not true without the additional term $\dip$ as long as $\omega$ is indefinite). 

With coefficients in $\Pe$,  direct and inverse spectral theory for~\eqref{eqnSP} shares a lot of similarities with the one for periodic Jacobi matrices \cite{mo76, tjac}.
In fact, this is not overly surprising as the differential equation~\eqref{eqnSP} reduces to a difference equation in this case.
However, there are a couple of subtle differences to point out. 
 First, one notes that it is possible to solve the differential equation~\eqref{eqnSP} explicitly when $z$ is equal to zero. 
 It follows readily  that zero is neither a periodic nor an antiperiodic eigenvalue, thus always creating a spectral gap around zero. 
 Unlike the remaining gaps, there is no Dirichlet eigenvalue present in this gap. 
 Secondly, there may or may not be Dirichlet eigenvalues above and below the spectral bands, which never happens for the auxiliary spectrum in the Jacobi case.

 \subsection*{Notation}  
  For integrals of a function $f$ which is locally integrable with respect to a Borel measure $\nu$ on an interval $I$, we will employ the convenient notation 
\begin{align}\label{eqnDefintmu}
 \int_x^y f\, d\nu = \begin{cases}
                                     \int_{[x,y)} f\, d\nu, & y>x, \\
                                     0,                                     & y=x, \\
                                     -\int_{[y,x)} f\, d\nu, & y< x, 
                                    \end{cases} \qquad x,\,y\in I, 
\end{align}
 rendering the integral left-continuous as a function of $y$. 
 If the function $f$ is locally absolutely continuous on $I$ and $g$ denotes a left-continuous distribution function of the Borel measure $\nu$, then we have the integration by parts formula 
\begin{align}\label{eqnPI}
  \int_{x}^y  f\, d\nu = \left. g f\right|_x^y - \int_{x}^y g(s) f'(s) ds, \quad x,\,y\in I,
\end{align}
 which will be used frequently throughout this article.

 \section{Floquet theory}\label{secDSP}

 Let us fix a pair $(u,\mu)$ in $\Pe$ throughout this section and first recall the definitions of the distribution $\omega$ in~\eqref{eqnDefomega} as well as the one of the Borel measure $\dip$ via~\eqref{eqnDefdip}. 
 As mentioned above, the distribution $\omega$ can be regarded as a Borel measure on $\R$ and upon choosing an arbitrary but fixed base point $a\in\R$, we may write 
 \begin{align}\label{eqnMEA}
 \omega & = \sum_{k\in\Z}\,\sum_{n=1}^N \omega_n \delta_{x_n + k\per}, &  \dip & = \sum_{k\in\Z}\, \sum_{n=1}^N \dip_n \delta_{x_n+k\per},
\end{align}
 for some non-negative integer $N$, strictly increasing points $x_1,\ldots,x_N\in [a,a+\per)$ and $\omega_n\in\R$, $\dip_n\geq 0$ for each $n\in\lbrace 1,\ldots,N\rbrace$.  
 In order to make this representation unique, we will always suppose that $|\omega_n|+\dip_n>0$ for all $n\in\lbrace 1,\ldots,N\rbrace$.  

 We are going to study spectral problems associated with the differential equation 
\begin{align}\label{eqnDE}
 -f'' + \frac{1}{4} f = z\, \omega f + z^2 \dip f, 
\end{align}
 where $z\in\C$ is a spectral parameter. 
 Regarding the precise meaning and basic properties of this differential equation with measure coefficients we refer to \cite[Section~2]{CHPencil}.
 Due to the special form of the Borel measures $\omega$ and $\dip$ in our case, this notion of solution can be described as follows: A function $f$ is a solution of the differential equation~\eqref{eqnDE} if it satisfies 
 \begin{align}\label{eqnDEdiscr01}
   -f'' + \frac{1}{4} f = 0 
 \end{align}
 away from the discrete support $\Sigma$ of $|\omega|+\dip$, together with the interface condition  
 \begin{align}
  \begin{pmatrix} f(x-) \\ f'(x-) \end{pmatrix}  =
\begin{pmatrix} 1 & 0  \\  z\, \omega(\lbrace x\rbrace) + z^2 \dip(\lbrace x\rbrace)  & 1 \end{pmatrix}
\begin{pmatrix} f(x+) \\ f'(x+) \end{pmatrix}
\label{eqnDEdiscr02}
 \end{align}
 for all $x\in\Sigma$. 
 Note that in this case, the solution $f$ is in general not differentiable at the points $x\in\Sigma$. 
 However, for simplicity of notation, we will always uniquely extend the derivative $f'$ to all of $\R$ by requiring it to be left-continuous. 
 With this notation, we readily obtain the useful identity 
 \begin{align}\label{eqnDEint}
   \int_x^y f'(s)h'(s)ds + \frac{1}{4} \int_x^y f(s)h(s)ds = \left.f'h\right|_{x}^{y} + z\int_x^y fh\,d\omega + z^2 \int_x^y fh\, d\dip
 \end{align}
 for all $x$, $y\in\R$ as long as the function $h$ belongs to $H^1_\loc(\R)$.

 It is not difficult to see that initial value problems for the differential equation~\eqref{eqnDE} always have a unique solution (see \cite[Lemma 2.1]{CHPencil} for example). 
 Thus we may introduce a particular fundamental system of solutions $c(z,\redot)$, $s(z,\redot)$ with the initial conditions
\begin{align}\label{eq:fss}
c(z,a) & = s'(z,a) = 1, &  c'(z,a) & = s(z,a)=0,
\end{align}
at the point $a$ for every $z\in\C$. 
Note that when $z$ is zero, we have explicitly 
\begin{align}\label{eqnCSatzero}
c(0,x) & = \cosh\left(\frac{x-a}{2}\right), & s(0,x) & = 2 \sinh\left(\frac{x-a}{2}\right), \qquad x\in\R.
\end{align}
 The functions $c(\ledot,x)$, $s(\ledot,x)$ as well as their spatial derivatives $c'(\ledot,x)$, $s'(\ledot,x)$ are real polynomials for every $x\in\R$. 
Next, we define {\em the monodromy matrix} $M$ via 
 \be\label{eq:cM}
 M(z)=\begin{pmatrix}
 c(z,a+\per) & s(z,a+\per)\\
 c'(z,a+\per) & s'(z,a+\per) 
 \end{pmatrix},\quad z\in\C, 
 \ee
 as well as {\em the Floquet discriminant} $\Delta$  by 
 \be\label{eq:Delta}
 \Delta(z)=\frac{\tr\,M(z)}{2} =\frac{c(z,a+\per)+s'(z,a+\per)}{2},\quad z\in\C.
 \ee 
 It readily follows from \eqref{eqnDEint} that the Wronski determinant of two solutions to \eqref{eqnDE} is constant on $\R$ and hence $\det M(z) =1$ for all $z\in\C$.  The qualitative behavior of the polynomial $\Delta$ is captured by the following result. 
 
 \begin{lemma}\label{lemDelta}
 All zeros of the Floquet discriminant $\Delta$ are real, non-zero and simple.
 Moreover, if $\dot{\Delta}(\lambda)=0$ for some $\lambda\in\C$, then $|\Delta(\lambda)|\ge 1$ and $\Delta(\lambda)\ddot{\Delta}(\lambda)<0$.
 \end{lemma} 
 
 \begin{proof}
   Let us suppose that $\lambda\in\C$ is a zero of the Floquet discriminant $\Delta$.
   Since the determinant of $M(\lambda)$ equals one, this implies that $\I$ is an eigenvalue of the matrix $M(\lambda)$. 
   Thus there is a nontrivial solution $f$ of the differential equation~\eqref{eqnDE} such that $f(a+\per) = \I f(a)$ and $f'(a+\per) = \I f'(a)$. 
   Setting $h=\lambda^\ast f^\ast$, $x=a$ and $y=a+\per$ in~\eqref{eqnDEint} and taking the imaginary part gives 
   \begin{align}\label{eqnfSA}
     (\lambda-\lambda^\ast) \left( \int_a^{a+\per} |f'(s)|^2ds + \frac{1}{4} \int_a^{a+\per} |f(s)|^2ds + \int_a^{a+\per} |\lambda f|^2 d\dip\right)=0.
   \end{align}
   Now we see that if $\lambda$ was non-real, then this would yield a contradiction. 
   
    In order to compute the derivative of the Floquet discriminant $\Delta$, we first introduce the solutions $c_\per(z,\redot)$, $s_\per(z,\redot)$ of the differential equation~\eqref{eqnDE} given by 
    \begin{align}\begin{split}\label{eqncsper}
      c_\per(z,x) & = s'(z,a+\per)c(z,x) - c'(z,a+\per)s(z,x), \\ s_\per(z,x) & = -s(z,a+\per)c(z,x) + c(z,a+\per)s(z,x), 
    \end{split}\end{align}
    for every $x\in\R$ and $z\in\C$, so that they satisfy the initial conditions  
   \begin{align*}
     c_\per(z,a+\per) & = s_\per'(z,a+\per) = 1, &  c_\per'(z,a+\per) & = s_\per(z,a+\per)=0,
   \end{align*}
   at the point $a+\per$.
   Upon choosing $f=c(z,\redot)$ and the function $h$ to be constant and equal to one, we differentiate~\eqref{eqnDEint} with respect to $z$ to obtain\footnote{The differentiation with respect to the spectral parameter is always done last.}
\begin{align*}
  \frac{1}{4} \int_x^y \dot{c}(z,s)ds & = \left. \dot{c}'(z,\redot)\right|_x^y + z\int_x^y \dot{c}(z,s)d\omega(s) + z^2 \int_x^y \dot{c}(z,s)d\dip(s) \\ 
   & \qquad\qquad\qquad\qquad\qquad\qquad\qquad\qquad + \int_x^y c(z,s)d\rho_z(s)
 \end{align*}
for every $x$, $y\in\R$, where $\rho_z$ is short for the Borel measure $\omega+2z\dip$. 
   By employing the integration by parts formula~\eqref{eqnPI}, this gives the identity
   \begin{align*}
    \dot{c}(z,a+\per) & = \left. \dot{c}(z,\redot) s_\per'(z,\redot) - \dot{c}'(z,\redot) s_\per(z,\redot) \right|_a^{a+\per} \\ 
                               & = \int_a^{a+\per} c(z,x)s_\per(z,x)d\rho_z(x), \quad z\in\C.
   \end{align*} 
   Furthermore, in much the same manner we also obtain
   \begin{align*}
    -\dot{s}'(z,a+\per) & =  \int_a^{a+\per} c_\per(z,x)s(z,x)d\rho_z(x), \quad z\in\C.
   \end{align*}   
   After plugging in~\eqref{eqncsper}, these two equations add up to
   \begin{align}\begin{split}\label{eqnDeltadot}
     \dot{\Delta}(z) & = \frac{c'(z,a+\per)}{2} \int_a^{a+\per} s(z,x)^2 d\rho_z(x)  - \frac{s(z,a+\per)}{2} \int_a^{a+\per} c(z,x)^2 d\rho_z(x) \\
                              & \qquad\quad + \frac{c(z,a+\per) - s'(z,a+\per)}{2} \int_a^{a+\per} c(z,x)s(z,x)d\rho_z(x), \quad z\in\C.
   \end{split}\end{align}
   Moreover, as long as $s(z,a+\per)$ is non-zero, we get 
   \begin{align}\label{eqnDeltadotpsi}
    \dot{\Delta}(z) & = -\frac{s(z,a+\per)}{2} \int_a^{a+\per} \psi_-(z,x) \psi_+(z,x) d\rho_z(x), 
   \end{align} 
   where $\psi_-(z,\redot)$ and $\psi_+(z,\redot)$ are the (nontrivial) solutions of~\eqref{eqnDE} given by
   \begin{align*}
    \psi_\pm(z,x) = c(z,x) + \frac{\Delta(z) \pm \sqrt{\Delta(z)^2-1} - c(z,a+\per)}{s(z,a+\per)} s(z,x), \quad x\in\R.
   \end{align*}    
   Now suppose that $\dot{\Delta}(\lambda)=0$ for some $\lambda\in\C$ (which then necessarily belongs to $\R$ since all zeros of $\Delta$ are real). 
   If $s(\lambda,a+\per)$ is zero, then 
   \begin{align*}
     c(\lambda,a+\per)s'(\lambda,a+\per) = \det M(\lambda) = 1
    \end{align*}
     and 
   \begin{align*}
    \left| c(\lambda,a+\per) + s'(\lambda,a+\per)\right| =  \left| c(\lambda,a+\per) + c(\lambda,a+\per)^{-1} \right| \geq 2.
   \end{align*}
   Otherwise, if $s(\lambda,a+\per)$ is non-zero and we suppose that $|\Delta(\lambda)|$ is less than one, then we have $\psi_-(\lambda,\redot)=\psi_+(\lambda,\redot)^\ast$ and the integral in~\eqref{eqnDeltadotpsi} turns into
   \begin{align*}
                \int_a^{a+\per} \left|\psi_+'(\lambda,x)\right|^2 dx + \frac{1}{4} \int_a^{a+\per} \left|\psi_+(\lambda,x)\right|^2 dx + \lambda^2 \int_a^{a+\per} \left|\psi_+(\lambda,x)\right|^2 d\dip(x) > 0
   \end{align*}   
   after an integration by parts.    
   Since this is a contradiction, we see that $\dot{\Delta}(\lambda)=0$ always implies $|\Delta(\lambda)|\geq1$. 
   The remaining claim is true for any real polynomial with only real and simple zeros. 
 \end{proof}

Let us now consider the spectral problem associated with our differential equation~\eqref{eqnDE} on the interval $[a,a+\per)$ with periodic/antiperiodic boundary conditions.
The corresponding periodic/antiperiodic spectrum $\sigma_{\pm}$ consist of all those $z\in\C$ for which there is a nontrivial solution $f$ of the differential equation~\eqref{eqnDE} with 
\begin{align}\label{eqnPBC}
 \begin{pmatrix} f(a) \\ f'(a) \end{pmatrix} = \pm \begin{pmatrix} f(a+\per) \\ f'(a+\per) \end{pmatrix}. 
\end{align}
Under the multiplicity of a periodic/antiperiodic eigenvalue we understand the number of linearly independent solutions of~\eqref{eqnDE} that satisfy~\eqref{eqnPBC}.

  \begin{proposition}\label{propperspec}
   The periodic/antiperiodic spectrum $\sigma_\pm$ is a finite set of nonzero reals and coincides with the set of zeros of the polynomial $\Delta\mp1$. 
   Each periodic/anti\-periodic eigenvalue's multiplicity is equal to its multiplicity as a zero of $\Delta\mp1$. 
  \end{proposition}
  
  \begin{proof}
   Let $\lambda\in\sigma_\pm$ be a periodic/antiperiodic eigenvalue with a corresponding eigenfunction $f$.
   It follows readily from~\eqref{eqnCSatzero} that $\lambda$ has to be non-zero. 
   Moreover, upon setting $h=\lambda^\ast f^\ast$, $x=a$ and $y=a+\per$ in~\eqref{eqnDEint} and taking the imaginary part gives~\eqref{eqnfSA}, which shows that $\lambda$ has to be real. 
   Furthermore, it is readily seen that some $z\in\C$ is a periodic/antiperiodic eigenvalue if and only if $\pm 1$ is an eigenvalue of the matrix $M(z)$, which is equivalent to $\Delta(z) = \pm1$. 
   
   Since the zeros of $\Delta\mp1$ have multiplicity at most two by Lemma~\ref{lemDelta}, it remains to show that a periodic/antiperiodic eigenvalue $\lambda\in\sigma_\pm$ is double if and only if $\dot{\Delta}(\lambda)$ vanishes. 
   If $\lambda$ has multiplicity two, then $\pm M(\lambda)$ is the identity matrix and~\eqref{eqnDeltadot} shows that $\dot{\Delta}(\lambda)$ is zero. 
   Conversely, if we suppose that $\dot{\Delta}(\lambda)$ vanishes, then certainly $s(\lambda,a+\per)$ is zero, since otherwise~\eqref{eqnDeltadotpsi} would give a contradiction. 
   It then follows readily that also $c(\lambda,a+\per)=s'(\lambda,a+\per)=\pm1$ and consequently~\eqref{eqnDeltadot} implies that $c'(\lambda,a+\per)=0$ and thus $\pm M(\lambda)$ is the identity matrix. 
  \end{proof}
     
   One sees that the periodic spectrum $\sigma_+$ and the antiperiodic spectrum $\sigma_-$ are not independent of each other. 
  In fact, we readily infer from Proposition~\ref{propperspec} that each of these spectra (including multiplicities) determines the other. 

\begin{corollary}\label{cor:per}
 The periodic/antiperiodic spectrum $\sigma_\pm$ (including multiplicities) determines the antiperiodic/periodic spectrum $\sigma_\mp$ (including multiplicities).
\end{corollary}

\begin{proof}
In view of Proposition~\ref{propperspec}, the periodic/antiperiodic spectrum $\sigma_\pm$ (including multiplicities) determines the polynomial $\Delta$ since its value at zero is $\cosh(\per/2)$.
 Conversely, the polynomial $\Delta$ also determines the antiperiodic/periodic spectrum $\sigma_\mp$ (including multiplicities). 
\end{proof}
 
 \begin{remark}
  It is not difficult to see that the periodic/antiperiodic spectrum is independent of the chosen base point $a$ and thus so is the Floquet discriminant.
\end{remark} 
 
 Let us now consider the zeros of the polynomial $\Delta^2 - 1$, each of which is non-zero, real and has multiplicity at most two in view of Proposition~\ref{propperspec}. 
 Since $\Delta^2-1$ is positive at zero and far out on the real axis, we may conclude that there is an even number of positive zeros as well as an even number of negative zeros (both counted with multiplicities). 
 Thus we may label them in non-decreasing order
 \begin{align}
  \lambda_{-2I_-},\lambda_{-2I_-+1},\ldots,\lambda_{-1},\lambda_1,\ldots,\lambda_{2I_+ -1},\lambda_{2I_+},
 \end{align}
 for some non-negative integers $I_-$, $I_+$ such that each zero has the same sign as its index. 
 It follows from Lemma~\ref{lemDelta} that this sequence indeed satisfies the inequalities  
\begin{align}\begin{split}
-\infty & <\lambda_{-2I_-} < \lambda_{-2I_-+1} \le \dots < \lambda_{-3} \le \lambda_{-2}  < \lambda_{-1} < 0 \\
 & \qquad\qquad\qquad 0 < \lambda_1 < \lambda_2 \le  \lambda_3 < \dots \le \lambda_{2I_+-1}<\lambda_{2I_+}<\infty ,
\end{split}\end{align}
where the lowest positive and negative zero is a simple periodic eigenvalue, followed by alternating pairs (except for the last one) of antiperiodic and periodic eigenvalues.  
 Upon introducing the index set $\inds = \lbrace -I_-,\ldots,-1\rbrace\cup\lbrace1,\ldots,I_+\rbrace$, we define the intervals
 \begin{align}
 \Gamma_i=\begin{cases} [-\infty,\lambda_{2i}], & i= -I_-, \\
                                           [\lambda_{2i-1},\lambda_{2i}], & i\in\lbrace -I_-+1,\ldots,-1\rbrace, \\
                                            [\lambda_{2i},\lambda_{2i+1}], & i\in\lbrace1,\ldots,I_+-1\rbrace, \\
                                             [\lambda_{2i},\infty], & i=I_+, \\  \end{cases} 
 \end{align}
 for each $i\in\inds$,  called {\em the gaps}. 
 A gap $\Gamma_i$ is called closed if it reduces to a single point and open otherwise. 
 If they exist, the last positive gap $\Gamma_{I_+}$ is called {\em the outermost positive gap} and the last negative gap $\Gamma_{-I_-}$ is called {\em the outermost negative gap}. 
 Obviously, the outermost gaps are always open. 
 The typical behavior of the Floquet discriminant $\Delta$ and the location of the gaps relative to it is depicted below:
  
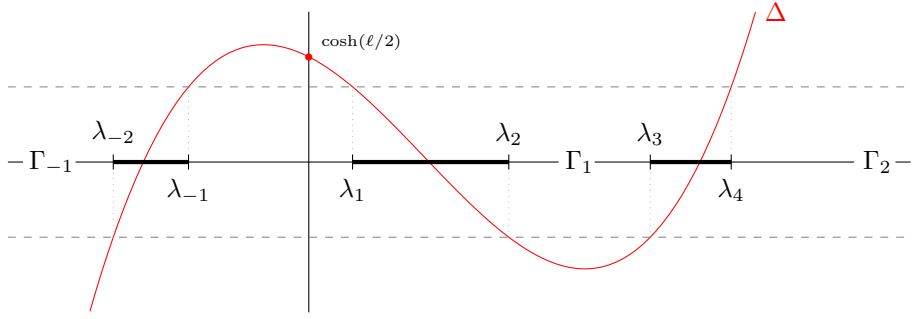
\begin{center}
 \begin{tikzpicture}[domain=-4:8, samples=101]

\draw[color=red, domain=-2.91:5.94] plot (\x,{ 1.4*(1-\x/1.6)*(1-\x/5.2)*(1+\x/2.2) }) node[right] {$\Delta$};

\draw[-] (-4,0) -- (8,0) node[above] {};
\draw[dashed, help lines] (-4,1) -- (8,1) node[right] {};
\draw[dashed, help lines] (-4,-1) -- (8,-1) node[right] {};
\draw[-] (0,-2) -- (0,2) node[right] {};

\draw[-] (-2.599,-0.1) -- (-2.599,0.1) node[above] {$\lambda_{-2}$};
\draw[-] (-1.599,0.1) -- (-1.599,-0.1) node[below] {$\lambda_{-1}$};
\draw[-] (0.582,0.1) -- (0.582,-0.1) node[below] {$\lambda_1$};
\draw[-] (2.659,-0.1) -- (2.659,0.1) node[above] {$\lambda_2$};
\draw[-] (4.54,-0.1) -- (4.54,0.1) node[above] {$\lambda_3$};
\draw[-] (5.616,0.1) -- (5.616,-0.1) node[below] {$\lambda_4$};

\draw[dotted, help lines] (-2.599,-1) -- (-2.599,-0.1) node[right] {};
\draw[dotted, help lines] (-1.599,0.1) -- (-1.599,1) node[right] {};
\draw[dotted, help lines] (0.582,0.1) -- (0.582,1) node[right] {};
\draw[dotted, help lines] (2.659,-1) -- (2.659,-0.1) node[right] {};
\draw[dotted, help lines] (4.54,-1) -- (4.54,-0.1) node[right] {};
\draw[dotted, help lines] (5.616,0.1) -- (5.616,1) node[right] {};

\draw[-, ultra thick] (-2.599,0) -- (-1.599,0) node[right] {};
\draw[-, ultra thick] (0.582,0) -- (2.659,0) node[right] {};
\draw[-, ultra thick] (4.54,0) -- (5.616,0) node[right] {};

\draw[-, color=white, ultra thick] (-3.8,0) -- (-3.1,0) node[above] {};
\draw[-, color=white, ultra thick] (3.32,0) -- (3.8,0) node[above] {};
\draw[-, color=white, ultra thick] (7.25,0) -- (7.8,0) node[above] {};
\node at (-3.42,-0.009) {$\Gamma_{-1}$}; 
\node at (3.6,0) {$\Gamma_{1}$};
\node at (7.56,0) {$\Gamma_{2}$};

\node[color=red] at (0,1.4) {\tiny $\bullet$};
\node at (0.7,1.6) {\tiny $\cosh(\per/2)$};

\end{tikzpicture}
\end{center}
 
In the solution of the periodic inverse spectral problem, a particular role will be played by the set $\To$ defined by\footnote{We use the convention $\Delta(-\infty)=\Delta(\infty)=1$ in this definition.}  
\begin{align}\label{eqnTodef}
 \To & = \prod_{i\in\inds} \To_i, & \To_i & = \left\lbrace (z,\zeta)\in \Gamma_i\times\R \left|\, \zeta^2 = \Delta(z)^2 - 1 \right. \right\rbrace, \quad i\in\inds.
\end{align}
Unless the gap $\Gamma_i$ is closed, the set $\To_i$ is homeomorphic to a torus (if $\Gamma_i$ is an outermost gap, then we view $\To_i$ as the one-point compactification of its finite part). 
The sets $\To_i$ correspond to connected components of the real part of a certain compactification of the underlying algebraic curve.  

 \section{Dirichlet spectrum}\label{sec:Dirichlet}

While continuing the notation from the previous section, we will next turn to the spectral problem associated with our differential equation~\eqref{eqnDE} on the interval $[a,a+\per)$ with Dirichlet boundary conditions at the endpoints.
The corresponding spectrum $\sigma$ consists of all those $z\in\C$ for which there is a nontrivial solution $f$ of the differential equation~\eqref{eqnDE} with $f(a)=f(a+\per)=0$.
From unique solvability of initial value problems for our differential equation, we see that such a solution is always unique up to scalar multiples. 
 
\begin{proposition}\label{prop:s}
 The Dirichlet spectrum $\sigma$ is a finite set of nonzero reals and coincides with the set of zeros of the polynomial $s(\ledot,a+\per)$, all of which are simple.  
\end{proposition}
 
\begin{proof}
 Let $\kappa\in\sigma$ be a Dirichlet eigenvalue with corresponding eigenfunction $f$. 
 It follows readily from~\eqref{eqnCSatzero} that $\kappa$ has to be non-zero. 
 Moreover, upon setting $h=\kappa^\ast f^\ast$, $x=a$ and $y=a+\per$ in~\eqref{eqnDEint} and taking the imaginary part gives~\eqref{eqnfSA}, which shows that $\kappa$ has to be real. 
 Furthermore, it is readily seen that some $z\in\C$ is a Dirichlet eigenvalue if and only if $s(z,a+\per)$ vanishes. 

In much the same way as in the proof of Lemma~\ref{lemDelta}, we obtain
\begin{align}\label{eqnsdot}
 \dot{s}(z,a+\per) = \int_a^{a+\per} s(z,x)s_\per(z,x) d\rho_z(x), \quad z\in\C, 
\end{align}
which shows that all zeros of $s(\ledot,a+\per)$ are simple upon using~\eqref{eqnDEint}. 
\end{proof}

The location of the Dirichlet spectrum in relation to the periodic and antiperiodic spectrum is now described by the following result. 

\begin{lemma}
 Each Dirichlet eigenvalue belongs to one of the gaps.  
 Conversely, except for the outermost gaps, each gap contains exactly one Dirichlet eigenvalue and each outermost gap contains at most one Dirichlet eigenvalue. 
\end{lemma}

\begin{proof}
 For every $z\in\C\backslash\R$, we consider the solution $\theta_\pm(z,\redot)$ of~\eqref{eqnDE} given by 
 \begin{align*}
  \theta_\pm(z,x) = \frac{s(z,a+\per)c(z,x) - (c(z,a+\per)\mp1)s(z,x)}{zs(z,a+\per)}, \quad x\in\R,
 \end{align*}
 and define the analytic function $M_\pm$ by  
 \begin{align*}
   M_\pm(z) = \theta_\pm'(z,a)z^\ast\theta_\pm(z,a)^\ast - \theta_\pm'(z,a+\per)z^\ast\theta_\pm(z,a+\per)^\ast, \quad z\in\C\backslash\R.
 \end{align*}
 Setting $f=\theta_\pm(z,\redot)$, $h=z^\ast\theta_\pm(z^\ast,\redot)$, $x=a$ and $y=a+\per$ in~\eqref{eqnDEint}, we obtain 
 \begin{align*}
  \frac{M_\pm(z)-M_\pm(z)^\ast}{z-z^\ast}  & = \int_a^{a+\per} \left|\theta_\pm'(z,x)\right|^2 dx +  \frac{1}{4} \int_a^{a+\per} \left|\theta_\pm(z,x)\right|^2 dx \\
                        & \qquad\qquad\qquad\qquad\quad  + \int_a^{a+\per} \left|z\theta_\pm(z,x)\right|^2 d\dip(x), \quad z\in\C\backslash\R.
 \end{align*}
 Thus the function $M_\pm$ is a Herglotz--Nevanlinna function and upon noting that 
 \begin{align*}
   M_\pm(z) = -\frac{2\Delta(z)\mp2}{zs(z,a+\per)}, \quad z\in\C\backslash\R,
 \end{align*}
 the claims follow from the corresponding interlacing property of zeros and poles.
\end{proof}

Let us now define a strictly increasing sequence $\kappa_i\in\Gamma_i$, indexed by $i\in\inds$, 
\begin{align}
 \kappa_{-I_-}, \kappa_{-I_-+1},\ldots,\kappa_{-1},\kappa_1,\ldots, \kappa_{I_+-1},\kappa_{I_+}
\end{align}
in the following way: 
For every $i\in\inds$ such that there is a Dirichlet eigenvalue in the gap $\Gamma_i$, we define $\kappa_i$ to be this (unique) eigenvalue.
If there is no Dirichlet eigenvalue in the gap $\Gamma_i$, then we define $\kappa_i$ to be $-\infty$ if $\Gamma_i$ is the outermost negative gap and $\kappa_i$ to be $\infty$ if $\Gamma_i$ is the outermost positive gap.
Clearly, we have 
\begin{align}
 -\infty \leq \kappa_{-I_-} < \dots < \kappa_{-1}< 0 < \kappa_1 < \dots < \kappa_{I_+} \leq \infty
\end{align}
and the Dirichlet spectrum $\sigma$ consists precisely of all those $\kappa_i$ that are finite.

\begin{remark}\label{rem:Dirspec}
 Let us point out that in contrast to the periodic and the antiperiodic spectrum, the Dirichlet spectrum in general does depend on the choice of the base point $a$ (as Proposition~\ref{propTID} will show).  
\end{remark}

Associated with the Dirichlet spectral problem is also a sequence of so-called {\em norming constants} $\nrc_\kappa$. 
For every Dirichlet eigenvalue $\kappa\in\sigma$, they are defined by 
\begin{align}\label{eqnnrc}
 \frac{1}{\nrc_\kappa} = \int_a^{a+\per} s'(\kappa,x)^2 dx + \frac{1}{4} \int_a^{a+\per} s(\kappa,x)^2 dx + \kappa^2 \int_a^{a+\per} s(\kappa,x)^2 d\dip(x)>0.
\end{align}
Upon employing~\eqref{eqnDEint} one more time, we see that 
\begin{align}\label{eqnnrc2}
 \frac{1}{\kappa\nrc_\kappa} = \int_a^{a+\per} s(\kappa,x)^2 d\omega(x) + 2\kappa \int_a^{a+\per} s(\kappa,x)^2 d\dip(x).
\end{align}
In addition, we furthermore introduce {\em the Dirichlet divisors} $\hat{\kappa}\in\To$ by  
\begin{align}\label{eqnhatkappa}
 \hat{\kappa}_i = \left(\kappa_i,  \Delta(\kappa_i) - s'(\kappa_i,a+\per) \right)\in\To_i, \quad i\in\inds.
\end{align}
Here the second component of $\hat{\kappa}_i$ is supposed to be interpreted as zero when $\Gamma_i$ is an outermost gap that does not contain a Dirichlet eigenvalue.

Another useful object in connection with the Dirichlet spectral problem on the interval $[a,a+\per)$ is {\em the Weyl--Titchmarsh function} $m$ given by 
\begin{align}\label{eq:m_dir}
 m(z)=-\frac{c(z,a+\per)}{zs(z,a+\per)},\quad z\in\C\backslash\R.
\end{align}
The decisive property of this function for solving the inverse spectral problem in the following section lies in the fact that it has a finite continued fraction expansion in terms of the coefficients $\omega$ and $\dip$.
 To this end, we introduce the quantities
 \begin{align}\label{eqnan}
  l_n & = 2 \tanh\left(\frac{x_{n}-a}{2}\right) - 2 \tanh\left(\frac{x_{n-1}-a}{2}\right), \quad n= 1,\ldots,N+1, 
  \end{align}
  where we set $x_0=a$ and $x_{N+1}=a+\per$ for convenience, as well as the polynomials 
 \begin{align}\label{eqnbn}
  q_n(z) & = \left(\omega_n + z\dip_n\right) \cosh^2\left(\frac{x_n-a}{2}\right), \quad z\in\C,~ n=1,\ldots,N.
 \end{align}

 \begin{lemma}\label{lemCF}
  The Weyl--Titchmarsh function $m$ admits the finite continued fraction expansion   
  \begin{align}\label{eqnContFrac}
    m(z) =  \cfrac{1}{-l_{1}z + \cfrac{1}{q_1(z) + \cfrac{1}{\;\ddots\; + \cfrac{1}{-l_Nz + \cfrac{1}{q_N(z) + \cfrac{1}{-l_{N+1}z}}}}}} \qquad z\in\C\backslash\R.
  \end{align}
 \end{lemma}

 \begin{proof}
 For each fixed $z\in\C\backslash\R$, consider the function $\Lambda(z,\cdot\,)$ on $[a,a+\per)$ given by 
  \begin{align*}
     \Lambda(z,x) = \frac{s_\per'(z,x)}{s_\per(z,x)} \cosh^2\left(\frac{x-a}{2}\right) - \frac{1}{2} \sinh\left(\frac{x-a}{2}\right) \cosh\left(\frac{x-a}{2}\right)
  \end{align*}
  for every $x\in[a,a+\per)$.
  Because of the interface condition \eqref{eqnDEdiscr02}, we first note that
  \begin{align*}
    \Lambda\left(z,x_n\right)-\Lambda\left(z,x_n+\right)=(z\,\omega_n+z^2\dip_n)\cosh^2\left(\frac{x_n-a}{2}\right), \quad n=1,\ldots,N.
  \end{align*}
  Upon differentiating the function $\Lambda(z,\redot)$, one furthermore obtains 
  \begin{align*}
   \Lambda'(z,x) = -\Lambda(z,x)^2 \cosh^{-2}\left(\frac{x-a}{2}\right),  
  \end{align*}
  for all $x\in(a,a+\per)$ away from $\lbrace x_1,\ldots,x_N\rbrace$, which immediately implies that 
  \begin{align*}
   \frac{1}{\Lambda(z,x_{n})} - \frac{1}{\Lambda(z,x_{n-1}+)} & =  l_{n}, & - \frac{1}{\Lambda(z,x_N+)} & = l_{N+1},
  \end{align*}
  for all $n\in\lbrace 2,\ldots,N\rbrace$, but also for $n=1$ if $x_1>a$. 
  Thus, we finally arrive at  
  \begin{align*}
    \frac{1}{\Lambda\left(z,x_{n-1}+\right)} = - l_{n} + \frac{1}{zq_{n}(z) + \Lambda\left(z,x_{n}+\right)}, \quad n=2,\ldots,N, 
  \end{align*} 
  and it remains to note that  
  \begin{align*}
   \frac{1}{zm(z)} = \frac{s_\per(z,a)}{s_\per'(z,a)} = \frac{1}{\Lambda\left(z,a\right)} = - l_{1} + \frac{1}{zq_{1}(z) + \Lambda\left(z,x_{1}+\right)}, 
   \end{align*}
  in order to  conclude the proof.  
 \end{proof}  
  
   On the other side, the Weyl--Titchmarsh function $m$ also has a particular partial fraction expansion in terms of the Dirichlet spectral data.
  
\begin{lemma}\label{lemmherglotz}
 The Weyl--Titchmarsh function $m$ is a rational Herglotz--Ne\-van\-linna function and admits the partial fraction expansion
 \begin{align}\label{eqnmIRep}
  m(z) = z\,\dip(\lbrace a\rbrace) + \omega(\lbrace a\rbrace) - \frac{1}{2z} \coth(\per/2) + \sum_{\kappa\in\sigma} \frac{\nrc_\kappa}{\kappa-z} , \quad z\in\C\backslash\R.  
 \end{align} 
\end{lemma}
  
\begin{proof}
 By setting $f=s_\per(z,\redot)$, $h=z^\ast s_\per(z^\ast,\redot)$, $x=a$ and $y=a+\per$ in~\eqref{eqnDEint}, we get
\begin{align*}
  \frac{m(z) - m(z)^\ast}{z-z^\ast}  
  & =  \int_a^{a+\per} \left|\frac{s_\per'(z,x)}{zs_\per(z,a+\per)}\right|^2 dx +\frac{1}{4} \int_a^{a+\per} \left|\frac{s_\per(z,x)}{zs_\per(z,a+\per)}\right|^2 dx \\ 
          & \qquad\qquad\qquad\qquad\qquad\quad + \int_{a}^{a+\per} \left|\frac{s_\per(z,x)}{s_\per(z,a+\per)}\right|^2 d\dip(x), \quad z\in\C\backslash\R,
 \end{align*}
 which shows that $m$ is a Herglotz--Nevanlinna function and it remains to verify the individual coefficients in the representation~\eqref{eqnmIRep}. 
 To this end, we first infer 
 \begin{align*}
  zm(z) =  z\,\omega(\lbrace a\rbrace) +  z^2\dip(\lbrace a\rbrace) + \OO(1), \qquad |z|\rightarrow\infty,  
 \end{align*}
 from the continued fraction expansion~\eqref{eqnContFrac} in Lemma~\ref{lemCF}. 
 We are left to determine the residues of the poles of $m$, for which we employ the identities   
 \begin{align}\label{eqnmres}
  \frac{c(0,a+\per)}{s(0,a+\per)} & = \frac{1}{2}\coth(\per/2), & \dot{s}(\kappa,a+\per) s'(\kappa,a+\per) & = \frac{1}{\kappa\nrc_\kappa}, \quad \kappa\in\sigma,
 \end{align}
 where the latter follows from~\eqref{eqnsdot}. 
 \end{proof}
 
\begin{remark}\label{rem:muinfty}
The former results about the Weyl--Titchmarsh function $m$ allow us to determine precisely when outermost gaps contain Dirichlet eigenvalues. 
To this end, we first infer from the partial fraction expansion that  
\begin{align}\label{eqnmperasym}
  - \frac{2\Delta(z) \mp 2}{z s(z,a+\per)} = z\,\dip(\lbrace a\rbrace) + \omega(\lbrace a\rbrace) + \oo(1), \qquad |z|\rightarrow \infty.
\end{align}
By comparing the degrees of the polynomials in the numerator and the denominator, this shows that $\dip(\lbrace a\rbrace)$ is non-zero if and only if both outermost gaps exist and do not contain Dirichlet eigenvalues. 
 In this case we have 
 \begin{align}
  \dip(\lbrace a\rbrace) = - \frac{\cosh(\per/2)\mp1}{\sinh(\per/2)} \frac{\prod_{\kappa\in\sigma} \kappa}{\prod_{\lambda\in\sigma_\pm} \lambda}. 
 \end{align}
 Moreover, if $\dip(\lbrace a\rbrace)$ is zero,  then $\omega(\lbrace a\rbrace)$ is positive if and only if an outermost positive gap exists and does not contain a Dirichlet eigenvalue and $\omega(\lbrace a\rbrace)$ is negative if and only if an outermost negative gap exists and does not contain a Dirichlet eigenvalue. 
 This follows upon noting that the sign of the fraction for large real $z$ is determined by the numbers of positive and negative Dirichlet eigenvalues in relation to the numbers of positive and negative periodic/antiperiodic eigenvalues. 
 In these cases we have 
 \begin{align}
  \omega(\lbrace a\rbrace) = \frac{\cosh(\per/2)\mp1}{\sinh(\per/2)} \frac{\prod_{\kappa\in\sigma} \kappa}{\prod_{\lambda\in\sigma_\pm} \lambda}. 
 \end{align}
 Finally, let us also note that neither $\omega$ nor $\dip$ have mass in $a$ if and only if every gap contains a Dirichlet eigenvalue. 
 In this case we have 
  \begin{align}
    \frac{1}{2}\coth\left(\frac{x_{1}-a}{2}\right) + \frac{1}{2}\coth\left(\frac{a+\per - x_{N}}{2}\right) & =  \frac{\cosh(\per/2)\mp1}{\sinh(\per/2)} \frac{\prod_{\kappa\in\sigma} \kappa}{\prod_{\lambda\in\sigma_\pm} \lambda}, 
 \end{align}
 except for the trivial case when $\omega$ and $\dip$ vanish identically. 
\end{remark}

\section{Trace formulas}\label{sec:Trace}
 
 With the notation from the previous sections, we are now going to collect some trace formulas, which provide relations between the pair $(u,\mu)$ and the periodic/antiperiodic as well as the Dirichlet spectrum. 
 To this end, let us first enumerate the non-decreasing sequence of periodic/antiperiodic eigenvalues (including multiplicities) as  $\lambda_{i}^\pm$ with index $i\in\inds$.  
  By means of the product representation 
   \begin{align}\label{eqnDeltaprod}
   \Delta(z)^2 - 1 = \sinh^2(\per/2) \prod_{i\in\inds} \biggl(1-\frac{z}{\lambda_{i}^+}\biggr) \biggl(1-\frac{z}{\lambda_{i}^-}\biggr), \quad z\in\C,
  \end{align}   
 for the polynomial $\Delta^2-1$, we are now going to derive trace formulas which will reappear as conserved quantities for periodic multi-peakon solutions. 
 
  \begin{proposition}\label{proptraceformulas}
   The first two trace formulas are
   \begin{align}
    \sum_{i\in\inds}\frac{1}{\lambda_{i}^+} + \frac{1}{\lambda_{i}^-} & = 2 \coth(\per/2)\int_a^{a+\per} u(x) dx, \\
    \sum_{i\in\inds} \frac{1}{(\lambda_{i}^{+})^2} + \frac{1}{(\lambda_{i}^-)^2} & = \frac{2}{\sinh^2(\per/2)} \biggl(\int_a^{a+\per} u(x)dx\biggr)^2 + 4\coth(\per/2) \int_a^{a+\per} d\mu. 
 \end{align}
 \end{proposition} 
 
 \begin{proof}
  Differentiating~\eqref{eqnDeltaprod} with respect to $z$ and then evaluating at zero gives 
  \begin{align*}
    2\cosh(\per/2)\dot{\Delta}(0) = -\sinh^2(\per/2)\sum_{i\in\inds}\frac{1}{\lambda_{i}^+} + \frac{1}{\lambda_{i}^-}.
  \end{align*}
  In order to verify the first trace formula, it remains to employ~\eqref{eqnDeltadot} to obtain 
  \begin{align}
   \dot{\Delta}(0) 
                               =  -\sinh(\per/2) \int_a^{a+\per} u(x) dx.  
  \end{align}
  Furthermore, after a rather long computation we also get   
  \begin{align}
   \ddot{\Delta}(0) = \cosh(\per/2) \biggl(\int_a^{a+\per} u(x)dx\biggr)^2 - 2 \sinh(\per/2) \int_a^{a+\per} d\mu
  \end{align}
  upon first differentiating~\eqref{eqnDeltadot} with respect to $z$ and using the identities  
  \begin{subequations}\label{eqnFSdotzero}
  \begin{align}\begin{split}\label{eqncdotzero}
   \dot{c}(0,x) & = - \int_a^x 2\sinh\left(\frac{x-s}{2}\right) \cosh\left(\frac{s-a}{2}\right) d\omega(s) \\
    & = \cosh\left(\frac{x-a}{2}\right) \left(u(x)-u(a)\right) \\
    & \quad\qquad\qquad -  \sinh\left(\frac{x-a}{2}\right) \left(2u'(a) + \int_a^x u(s)ds\right), \quad x\in\R,
   \end{split}\end{align}
   \begin{align}\begin{split}
   \dot{c}'(0,x) & = - \int_a^x \cosh\left(\frac{x-s}{2}\right) \cosh\left(\frac{s-a}{2}\right) d\omega(s) \\
                      & = \cosh\left(\frac{x-a}{2}\right) \left(u'(x)-u'(a)-\frac{1}{2} \int_a^x u(s)ds\right) \\ 
                      & \qquad\qquad\qquad\qquad - \sinh\left(\frac{x-a}{2}\right) \frac{u(x)+u(a)}{2}, \quad x\in\R, 
   \end{split}\end{align}
   \begin{align}\begin{split}\label{eqnsdotzero}
   \dot{s}(0,x) & = -  \int_a^x 2\sinh\left(\frac{x-s}{2}\right) 2\sinh\left(\frac{s-a}{2}\right) d\omega(s) \\
                      & = 2\sinh\left(\frac{x-a}{2}\right) \left(u(x)+u(a)\right) \\
                      & \qquad\qquad\qquad\qquad - 2 \cosh\left(\frac{x-a}{2}\right) \int_a^x u(s)ds, \quad x\in\R, 
   \end{split}\end{align}
   \begin{align}\begin{split}
   \dot{s}'(0,x) & = - \int_a^x \cosh\left(\frac{x-s}{2}\right) 2\sinh\left(\frac{s-a}{2}\right) d\omega(s) \\
                      & = 2\sinh\left(\frac{x-a}{2}\right) \left( u'(x) - \frac{1}{2} \int_a^x u(s)ds \right) \\ 
                      & \qquad\qquad\qquad\qquad - \cosh\left(\frac{x-a}{2}\right) \left(u(x)-u(a)\right), \quad x\in\R.
  \end{split}\end{align}
  \end{subequations}
  Evaluating the second derivative of~\eqref{eqnDeltaprod} at zero gives the second trace formula. 
 \end{proof}
 
  Let us mention that we also easily obtain trace formulas for the periodic spectrum and trace formulas for the antiperiodic spectrum alone, by using the method employed in the proof of Proposition~\ref{proptraceformulas}. 
   
  In a similar manner, we are now going to exploit the product representation\footnote{We employ the convention that a fraction with $\pm\infty$ in the denominator is regarded to be zero.}
 \begin{align}\label{eqnF}
  s(z,a+\per) = 2 \sinh(\per/2) \prod_{i\in\inds} \biggl( 1-\frac{z}{\kappa_i}\biggr), \quad z\in\C,
 \end{align}
 to derive further identities that also involve the Dirichlet spectrum. 
 
 \begin{proposition}\label{propTID}
  We have the identities 
 \begin{align}\label{eqnDtrace}
    \frac{1}{4} \sum_{i\in\inds} \frac{1}{\lambda_{i}^+} +  \frac{1}{\lambda_{i}^-} - \frac{2}{\kappa_{i}} & = u(a), & 
   \frac{1}{16} \sum_{i\in\inds} \frac{1}{(\lambda_{i}^+)^2} +  \frac{1}{(\lambda_{i}^-)^2} - \frac{2}{\kappa_{i}^2} & = P(a), 
\end{align}
where $P$ is the periodic function given by 
  \begin{align}
  P(x) = \frac{1}{4}  \int_\R \E^{-|x-s|} u(s)^2 ds + \frac{1}{4} \int_\R \E^{-|x-s|} d\mu(s), \quad x\in\R.
 \end{align}
 \end{proposition}

\begin{proof}
 Upon employing~\eqref{eqnsdotzero} to compute
 \begin{align*}
  \dot{s}(0,a+\per) = 4\sinh(\per/2) u(a) - 2\cosh(\per/2) \int_a^{a+\per} u(x)dx, 
 \end{align*}
 the first identity follows by differentiating~\eqref{eqnF} with respect to $z$, evaluating at zero and taking into account the first trace formula in Proposition~\ref{proptraceformulas}. 
 After a much longer calculation, differentiating~\eqref{eqnsdot} with respect to $z$, evaluating at zero and plugging in~\eqref{eqnFSdotzero}, we eventually end up with 
  \begin{align*}
    \ddot{s}(0,a+\per) & = 2\sinh(\per/2)\biggl(\int_a^{a+\per} u(x)dx\biggr)^2 - 8 \cosh(\per/2) u(a)\int_a^{a+\per} u(x)dx \\ 
     & \qquad\quad - 4\cosh(\per/2) \int_a^{a+\per} d\mu  + 8 \sinh(\per/2)u(a)^2 + 16 \sinh(\per/2)P(a),
 \end{align*}
 where we also note that the function $P$ is a solution of the differential equation
  \begin{align*}
  P - P'' =\frac{u^2+\mu}{2}. 
 \end{align*} 
 Evaluating the second derivative of~\eqref{eqnF} at zero then gives the second identity.
\end{proof}

\section{Inverse spectral problems}

 The main purpose of the present section is to solve the inverse problem for the periodic/antiperiodic spectrum, based on the solution of the inverse problem for the Dirichlet spectrum on the interval $[a,a+\per)$ for a fixed base point $a\in\R$. 
    
  \begin{theorem}\label{thmIP}
   Let $\sigma$ be a finite set of nonzero reals and for each $\kappa\in\sigma$ let $\nrc_\kappa$ be a positive number.
    Then for every $\omega_a\in\R$ and $\dip_a\geq0$ there is a unique pair $(u,\mu)$ in $\Pe$ such that the associated Dirichlet spectrum coincides with $\sigma$, the norming constants are $\nrc_\kappa$ for each $\kappa\in\sigma$ as well as $\omega(\lbrace a\rbrace)=\omega_a$ and $\dip(\lbrace a\rbrace)=\dip_a$.
  \end{theorem}

 \begin{proof}
 {\em Existence.}
    It follows from \cite[Lemma~B]{UniSolCP} or  \cite[Corollary~4.3 and Corollary~4.4]{IndMoment} that the rational function $m$ defined by    
  \begin{align}\label{eq:mfunction}
    m(z) = z\,\dip_a + \omega_a - \frac{1}{2z} \coth\left(\per/2\right) + \sum_{\kappa\in\sigma} \frac{\nrc_\kappa}{\kappa-z} , \quad z\in\C\backslash\R,
  \end{align}     
 has a finite continued fraction expansion of the form~\eqref{eqnContFrac} for some non-negative integer $N$, some reals $l_1\geq0$ and $l_n>0$ for $n=2,\ldots,N+1$ and some real, nonzero polynomials $q_n$ of degree at most one with $\dot{q}_n(0)\geq0$ for every $n=1,\ldots,N$. 
 Because of the identity 
 \begin{align*}
    \sum_{n=1}^{N+1} l_n = - \lim_{z\rightarrow0} \frac{1}{zm(z)} = 2 \tanh(\per/2), 
 \end{align*}
 it is possible to define strictly increasing points $x_1,\ldots,x_N\in[a,a+\per)$ through 
 \begin{align*}
   2\tanh\left(\frac{x_n-a}{2}\right) =   \sum_{j=1}^{n} l_j, \quad n\in\lbrace 1,\ldots,N\rbrace,
 \end{align*} 
 so that~\eqref{eqnan} holds.
 Furthermore, we define weights $\omega_n\in\R$ and $\dip_n\geq0$ for every $n\in\lbrace 1,\ldots,N\rbrace$ in such a way that equation~\eqref{eqnbn} holds. 
 Upon defining the Borel measures $\omega$ and $\dip$ via~\eqref{eqnMEA} as well as the corresponding pair $(u,\mu)$ in $\Pe$, it follows from Lemma~\ref{lemCF} that the associated Weyl--Titchmarsh function from Section~\ref{sec:Dirichlet} coincides with $m$ by construction. 
 Thus, we see that the pair $(u,\mu)$ has all the claimed properties upon comparing with Lemma~\ref{lemmherglotz}.  
 
  {\em Uniqueness.}
  Lemma~\ref{lemmherglotz} shows that our given data uniquely determines the Weyl--Titchmarsh function from Section~\ref{sec:Dirichlet}.
  Because all the coefficients in the continued fraction expansion~\eqref{eqnContFrac} can be read off from the asymptotics as $|z|\rightarrow\infty$, this implies that the measures $\omega$ and $\dip$ are uniquely determined. 
   \end{proof}
  
  A convenient way to reformulate the previous result is by employing the Weyl--Titchmarsh function for the Dirichlet spectral problem introduced in Section~\ref{sec:Dirichlet}. 
  
  \begin{corollary}\label{corIPm}
  Let $m$ be a rational Herglotz--Nevanlinna function with residue
  \begin{align}
     - \frac{1}{2} \cosh(\per/2) 
  \end{align}
  at zero. 
  Then there is a unique pair $(u,\mu)$ in $\Pe$ such that the function $m$ is the associated Weyl--Titchmarsh function as defined in Section~\ref{sec:Dirichlet}. 
 \end{corollary}
 
 \begin{proof}
  It suffices to notice that as a Herglotz--Nevanlinna function, the function $m$ has a representation of the form~\eqref{eq:mfunction} for some $\dip_a\geq0$, $\omega_a\in\R$, a finite set $\sigma$ of nonzero reals and positive numbers $\nrc_\kappa$ for each $\kappa\in\sigma$. 
 \end{proof}
  
 By means of using Stieltjes-type formulas (see \cite{krla79} or \cite{ConservMP}, \cite{IndMoment}) for the coefficients in the continued fraction expansion in~\eqref{eqnContFrac}, it is possible to recover solutions of the inverse Dirichlet spectral problem explicitly in terms of the spectral data. 
 From this, one obtains the following fact which characterizes the subclass of $\Pe$ that gives rise to purely positive (negative) Dirichlet spectrum. 

\begin{remark}\label{remwvprob} 
Let $(u,\mu)$ be a pair in $\Pe$  with associated Dirichlet spectrum $\sigma$. 
 Then the corresponding Borel measure $\dip$ vanishes on $(a,a+\per)$ and the corresponding Borel measure $\omega$ is non-negative (non-positive) on $(a,a+\per)$ if and only if $\sigma$ is positive (negative). 
 \end{remark}

Based on the solution of the inverse Dirichlet spectral problem, we are now going to solve the periodic inverse spectral problem. 
To this end, let $\Delta$ be a polynomial with only real, non-zero and simple roots, normalized at zero by 
\begin{align}
 \Delta(0) = \cosh(\per/2). 
\end{align}
We will furthermore suppose that $\dot{\Delta}(\lambda)$ can be zero only if $|\Delta(\lambda)|\geq 1$. 
Under these assumptions, we are able to define an index set $\inds$ and intervals $\Gamma_i$ as in Section~\ref{secDSP} as well as an associated torus $\To$ by~\eqref{eqnTodef}.

\begin{theorem}\label{thmPIP}
 For every $\hat{\kappa}\in\To$ there is a unique pair $(u,\mu)$ in $\Pe$ such that the associated Floquet discriminant coincides with $\Delta$ and $\hat{\kappa}$ are the Dirichlet divisors.  
\end{theorem}

\begin{proof}
{\em Existence.}
 Let us write $\hat{\kappa}_i = (\kappa_i,\zeta_i)$ for  $i\in\inds$,  define the polynomial $\varsigma$ by 
 \begin{align*}
  \varsigma(z) = 2\sinh(\per/2) \prod_{i\in\inds} \left(1-\frac{z}{\kappa_{i}}\right), \quad z\in\C, 
 \end{align*}
 and constants $\dip_a$, $\omega_a\in\R$ with $\dip_a\geq0$ such that 
 \begin{align*}
  -\frac{2\Delta(z)-2}{z \varsigma(z)} = z\,\dip_a + \omega_a + \oo(1), \qquad |z|\rightarrow\infty.
 \end{align*}
 For all $i\in\inds$ such that $\kappa_i$ is finite, we introduce $\gamma_i$ by 
 \begin{align*}
  \frac{1}{\gamma_i} = \kappa_i \dot{\varsigma}(\kappa_i) \left(\Delta(\kappa_i)-\zeta_i\right).
 \end{align*} 
 Since all these quantities are positive, 
  we infer from Corollary~\ref{corIPm} that there is a pair $(u,\mu)$ in $\Pe$ such that the associated Weyl--Titchmarsh function is given by 
 \begin{align*}
  m(z) = z\,\dip_a + \omega_a - \frac{1}{2z} \cosh(\per/2) + \sum_{i\in\inds} \frac{\gamma_i}{\kappa_i-z}, \quad z\in\C\backslash\R.
 \end{align*}
 Let $c$ and $s$ be the corresponding fundamental system of solutions as in Section~\ref{secDSP}. 
 We note that the polynomial $s(\ledot,a+\per)$ coincides with $\varsigma$ and thus~\eqref{eqnmres} shows
 \begin{align*}
   s'(\kappa_i,a+\per)  = \Delta(\kappa_i) - \zeta_i
 \end{align*}
 when $\kappa_i$ is finite. 
 Now in view of~\eqref{eqnmperasym} we have 
 \begin{align*}
  -\frac{c(z,a+\per) + s'(z,a+\per) - 2}{z s(z,a+\per)} & = z\,\dip_a + \omega_a + \oo(1), \qquad |z|\rightarrow\infty,
 \end{align*}
 and the residue of the function on the left-hand side at every finite $\kappa_i$ is
 \begin{align*}
  -\frac{s'(\kappa_i,a+\per)^{-1} + s'(\kappa_i,a+\per) -2}{\kappa_i \dot{s}(\kappa_i,a+\per)} = -\frac{2\Delta(\kappa_i)-2}{\kappa_i\dot{\varsigma}(\kappa_i)}.
 \end{align*}
 Together, this implies that $\Delta$ is indeed the Floquet discriminant associated with the pair $(u,\mu)$. 
 It follows readily that $\hat{\kappa}$ are the corresponding Dirichlet divisors. 
 
 {\em Uniqueness.}
 We infer from~\eqref{eqnmperasym} that our spectral data uniquely determines the linear term in the representation~\eqref{eqnmIRep} of the Weyl--Titchmarsh function.
 The poles are uniquely determined too, as well as the corresponding residues in view of~\eqref{eqnmres} and the definition~\eqref{eqnhatkappa} of the Dirichlet divisors. 
 Thus so is the Weyl--Titchmarsh function and uniqueness follows from Corollary~\ref{corIPm}. 
\end{proof}

Similar to Remark~\ref{remwvprob}, we are able to relate properties about the pair $(u,\mu)$ to positivity (negativity) of the corresponding periodic/antiperiodic spectrum. 

\begin{remark}
 Let $(u,\mu)$ be a pair in $\Pe$ such that its associated Floquet discriminant coincides with $\Delta$.  
 Then the corresponding Borel measure $\dip$ vanishes identically and the corresponding Borel measure $\omega$ is non-negative (non-positive) if and only if all zeros of the polynomial $\Delta\mp1$ are positive (negative).
\end{remark}

We conclude this section with a result about the isospectral set $\Iso{\Delta}$, which is defined as the collection of all those pairs $(u,\mu)$ in $\Pe$ whose associated Floquet discriminant coincides with the polynomial $\Delta$. 

\begin{corollary}
By means of the mapping 
  \begin{align}\label{eqnHom}
    (u,\mu) & \mapsto \hat{\kappa}
  \end{align} 
 the isospectral set $\Iso{\Delta}$ can be identified with the torus $\To$.
\end{corollary}

\begin{proof}
This follows immediately from Theorem~\ref{thmPIP}.
\end{proof}

\end{document}